\documentclass[preprint, 11pt]{elsarticle}
\usepackage{bbm}

\usepackage{amssymb}
\usepackage{amsfonts}
\usepackage{mathrsfs}

\usepackage{color, amstext,amsthm}

\usepackage{epsfig, latexsym}
\usepackage{amsmath}
\usepackage{graphicx}
\usepackage{longtable}

\numberwithin{equation}{section}

\allowdisplaybreaks
\newtheorem{definition}{Definition}[section]
\newtheorem{theorem}{Theorem}[section]
\newtheorem{lemma}{Lemma}[section]
\newtheorem{remark}{Remark}[section]

\newtheorem{assumption}{Assumption}

\journal{arXiv}

\begin{document}

\begin{frontmatter}



\title{The impact of multiplicative noise in SPDEs close to bifurcation via amplitude equations. }


\author{Hongbo Fu}
\ead{hbfu@wtu.edu.cn}
\address{Research Center of Nonlinear Science, College of Mathematics and Computer Science, Wuhan Textile University, Wuhan, 430073, PR China}

\author{Dirk Bl\"{o}mker} 
\ead{dirk.bloemker@math.uni-augsburg.de}
\address{Institut f\"{u}r Mathematik, Universit\"{a}t Augsburg, 86135, Augsburg, Germany}

\begin{abstract}
This article deals with the  approximation of a stochastic
partial differential equation (SPDE) via  amplitude equations.
We consider an SPDE with a cubic nonlinearity 
perturbed by a general multiplicative noise that preserves 
the constant trivial solution  and we study the dynamics around it 
for the deterministic equation being close to a bifurcation.

Based on the separation of time-scales close to a
change of stability, we rigorously derive an amplitude equation
describing the dynamics of the bifurcating pattern.

This allows us to approximate the original infinite dimensional
dynamics by a simpler effective dynamics associated with the solution of the amplitude
equation. To illustrate the abstract result we apply it to a
simple one-dimensional stochastic Ginzburg-Landau equation.
\end{abstract}


\end{frontmatter}

\section{Introduction}

In this paper
we study a class of stochastic partial differential equations (SPDEs) of the following
form
\begin{eqnarray}\label{original-eqn}
du(t)=[\mathcal{A}u(t)+\varepsilon^2\mathcal
{L}u(t)+\mathcal{F}(u(t))]dt+\varepsilon G(u(t))dW(t),
\end{eqnarray}
where $\mathcal{A}$ is a non-positive self-adjoint operator with
finite-dimensional kernel,
$\varepsilon^2\mathcal {L}u(t)$ is a
small deterministic perturbation with
$\varepsilon>0$ measuring the distance to the change of stability.
The nonlinearity $\mathcal{F}$ is a cubic
mapping, and $G(u)$ is Hilbert-Schmidt operator with $G(0)=0$ so that the constant $u=0$ is a solution to equation \eqref{original-eqn}.
The noise is given via a (possibly infinite dimensional cylindrical) Wiener process  $W$ on some stochastic basis.

Our aim is to study in the limit $\varepsilon\to 0$
the asymptotic dynamics of the solution $u(t)$
to equation \eqref{original-eqn} on the natural slow
time-scale of order $\varepsilon^{-2}$.

Near a change of stability of the linearized operator
$\mathcal{A} + \varepsilon^2\mathcal{L}$, a natural separation of
time-scales allows  the original system to be transferred into slow
dynamics on a dominant pattern, which couples to dynamics on a fast
time scale. A reduced equation eliminating the fast variable and
characterizing the behavior of dominant modes significantly
simplifies the dynamics to a stochastic differential equation (SDE).
This equation identifies the amplitudes of dominant pattern and is
often called amplitude equation.

Amplitude approximation  plays a prominent role in qualitative
analysis on the dynamics of stochastic systems near a change
stability.
For additive noise amplitude approximation for SPDEs has been studied in many cases starting from
\cite{BMPS01-DCDS} and later  \cite{DBHa:04,DMG-2007-nonlinearity,DW-2009-EJP}.
See also \cite{KiFa:14,WaRo:13,PPKPT:12,KuLeRo:18} for related work.

For the case of SPDEs on
unbounded domains the effective equation is no longer an SDE,
but the reduced model is still given as an infinite dimensional  SPDE.
For details see \cite{WDK-2013-SIAM} in the case of a simple one-dimensional noise,
\cite{BlHaPa:05} for large domains and \cite{BiBlSc:17,KuLeRo:18} for
results with space-time white noise and on an unbounded domain.
Here we will focus on the case of bounded domains only.

Amplitude equations can be used to qualitatively describe
the dynamics close to a change of stability.
In \cite{DBHa:04} amplitude equations were used to give an approximation of the
infinite-dimensional invariant measure for a Swift-Hohenberg equation,
while in \cite{DBHa:04,DBHa:05,DB:book} ideas were presented that would allow
to approximate random attractors or random invariant manifolds via amplitude equations.
See also \cite{BiBlYa:16,KiFa:14} for results for other models with additive noise.

While many results for the approximation via amplitude equations
were established for additive noise, the case of multiplicative noise is not that well studied.
Only for the very special case of $G(u)=u$ and $W$ being a scalar Brownian-motion
first results for amplitude equations were obtained in \cite{DB:book}.
With this special case of noise the
approximation of random invariant manifolds was studied in \cite{DBWW:10}
In first approximation the dynamics on the dominant space is given by a variant of the amplitude equation,
while for the qualitative description of a random invariant manifold, one also needs an effective
equation on the infinite dimensional remainder. See also \cite{GuDuLi:16,SuDuLi:10}
or \cite{LiHeOn:18} using parameterizing manifolds introduced by \cite{ChLiWa:15}, see also \cite{CLPR:19}.


In the present paper our main contribution is the analysis in the case of general
infinite-dimensional multiplicative noise.
We will only treat the case with $G(0)=0$, so there is no contribution
by an additive noise, which would lead to a different scaling in $\varepsilon$ of the noise. 

Under some smoothness assumptions  on the diffusion
coefficient $G$ and regularity conditions on the noise, we derive
the amplitude equations of responding equation
\eqref{original-eqn} and show rigorously, that it
captures the essential dynamics of the
dominant modes.
We use the Taylor
expansion of $G$ in order to directly determine the errors bounds
between the solution of \eqref{original-eqn}  and that of the amplitude
equation which is only on the dominant modes.

The organization of this paper is as follows: In Section 2, we
formulate the abstract framework and some basic assumptions. Section
3 contains the main results of the paper as presented in Theorem
\ref{Theorem}. In Section 4, we give the proof of the main results.
Section 5 is devoted to a illustrative example.

\section{Setting and assumptions}
Throughout the paper, we shall work in a separable Hilbert space
$\mathcal {H}$, endowed with the usual scalar product $\langle
\cdot, \cdot\rangle$ and with the corresponding norm $\|\cdot\|$.
Concerning the leading operator $\mathcal {A}$ we shall assume the
following conditions.

\begin{assumption}[Linear Operator $\mathcal{A}$]
Suppose $\mathcal{A}$ is a self-adjoint and non-positive operator on
$\mathcal {H}$ with   eigenvalues $\{-\lambda_k\}_{k\in\mathbb{N}}$
such that $0=\lambda_1\leq \cdots \leq\lambda_k\cdots$, satisfying
$\lambda_k\geq Ck^m$ for all sufficiently large $k$, positive
constants $m$ and $C$. The associated eigenvectors
$\{e_k\}_{k=1}^\infty$ form  a complete orthonormal basis in
$\mathcal{H}$ such that $\mathcal{A}e_k=-\lambda_ke_k.$

By $\mathcal{N}$ we denote the kernel space of $\mathcal{A}$, which,
according to Assumption 1, has finite dimension $n$ with basis
$\{e_1,\cdots, e_n\}.$ By $P_c$ we denote the orthogonal projector
from $\mathcal{H}$ onto $\mathcal{N}$ with respect to the inner
product $\langle \cdot, \cdot\rangle$, and by $P_s$ the orthogonal
projector from $\mathcal{H}$ onto the orthogonal complement
$\mathcal{N}^\perp$.
\end{assumption}

One standard example is with $m=4/d$ is the Swift-Hohenberg operator
$\mathcal{A}=-(1+\Delta)^2$ on $\mathcal {H}=L^2([-\pi,\pi]^d)$ subject
to periodic boundary conditions. Similar is the Laplacian
$\Delta$ with $m=2/d$. But we could also treat more general equations and also
coupled systems of SPDEs here.

\begin{remark}
Let us remark that the setting of a Hilbert space and $\mathcal{A}$ being
a self-adjoint operator is mainly for simplicity of presentation,
as many crucial properties about the $\mathcal{H}^\alpha$-spaces defined below,
the projections $P_c$ and $P_s$
and the semigroup $e^{t\mathcal{A}}$ generated by $ \mathcal{A}$ follow
in this setting as trivial Lemmas.
Otherwise we would need to formulate them as an assumption and verify them in the given application.
\end{remark}

We can now define fractional Sobolev-spaces $\mathcal{H}^\alpha=D((1-\mathcal{A})^{\alpha/2})$ by using the domain of
definition of fractional powers of the operator $\mathcal{A}$:
\begin{definition}
For $\alpha\in\mathbb{R}$, we define the space $\mathcal{H}^\alpha$
as
\begin{eqnarray*}
\mathcal{H}^\alpha=\left\{\sum\limits_{k=1}^\infty \gamma_ke_k:
\sum\limits_{k=1}^\infty \gamma_k^2(\lambda_k+1)^{\alpha}<\infty
\right\},
\end{eqnarray*}
which is endowed with the norm
\begin{eqnarray*}
\Big\|\sum\limits_{k=1}^\infty
\gamma_ke_k\Big\|_\alpha=\Big(\sum\limits_{k=1}^\infty
\gamma_k^2(\lambda_k + 1)^{\alpha}\Big)^{\frac{1}{2}}.
\end{eqnarray*}
\end{definition}
The operator $\mathcal{A}$ generates an analytic semigroup
$\{e^{\mathcal{A}t}\}_{t\geq 0}$ on any space $\mathcal{H}^\alpha$,
 defined by $$e^{\mathcal{A}t}\Big(\sum\limits_{k=1}^\infty
\gamma_ke_k\Big)=\sum\limits_{k=1}^\infty e^{-\lambda_k,
t}\gamma_ke_k, \;t\geq0,$$ and admits the following estimate, which
is a classical property for an analytic semigroup. Its proof is straightforward and omitted here.
\begin{lemma}
Under Assumption 1, for all $\beta\leq\alpha$, $\rho\in(\lambda_n,
\lambda_{n+1}]$, there exists a constant $M>0$, which is
independent of $u\in \mathcal{H}$, such that for any $t>0$
\begin{eqnarray}
\Big\|e^{\mathcal{A}t}P_su\Big\|_\alpha\leq
Mt^{-\frac{\alpha-\beta}{m}}e^{-\rho t}\Big\|P_su\Big\|_\beta.
\label{semigroup}
\end{eqnarray}
\end{lemma}
In addition,  we impose the following conditions:

\begin{assumption}[Operator $\mathcal{L}$]
  Let
$\mathcal{L}: \mathcal{H}^\alpha\rightarrow
\mathcal{H}^{\alpha-\beta}$ for some $\alpha\in \mathbb{R}$,
$\beta\in [0, m)$ be a linear
continuous mapping that commutes with $P_c$ and $P_s$.
\end{assumption}

\begin{assumption}[Nonlinearity $\mathcal{F}$] Assume
that $\mathcal{F}: (\mathcal{H}^\alpha)^3\rightarrow
\mathcal{H}^{\alpha-\beta}$, with $\alpha$ and $\beta$ as in
Assumption 2, is a trilinear, symmetric mapping and satisfies the
following conditions, for some $C>0$,

\begin{eqnarray}\label{F-condition-1}
\big\|\mathcal{F}(u,v,w)\big\|_{\alpha-\beta}\leq
C\big\|u\big\|_\alpha\big\|v\big\|_\alpha\big\|w\big\|_\alpha\quad\text{ for all }\;u,v,w\in
\mathcal{H}^\alpha.
\end{eqnarray}
Moreover, we have 
\begin{eqnarray}\label{F-condition-2}
\langle\mathcal{F}_c(u), u\rangle\leq 0 \quad\text{ for all }u\in \mathcal{N},
\end{eqnarray}
\begin{eqnarray}\label{F-condition-3}
\langle\mathcal{F}_c(u, u, w),w\rangle\leq 0 \quad\text{ for all }u, w\in \mathcal{N},
\end{eqnarray}
and for some positive constants $C_0$, $C_1$, and $C_2$ we have for all $u,v, w\in \mathcal{N}$ that
\begin{eqnarray}\label{F-condition-4}
\langle\mathcal{F}_c(u, v, w)-\mathcal{F}_c(v),u\rangle\leq
-C_0\|u\|^4+C_1\|w\|^4+C_2\|w\|^2\|v\|^2 .
\end{eqnarray}
\end{assumption}
Here, to ease notation, we use $\mathcal{F}_c:=P_s\mathcal{F}$ and we
define $\mathcal{F}_s$, $\mathcal{L}_c$ and $\mathcal{L}_s$ in a
similar way.
\begin{assumption}[Wiener Process]
 Let $U$ be a
separable Hilbert space with inner product $\langle\cdot,\cdot
\rangle_U$. Let $\{W_t\}_{t\geq0}$ be the cylindrical Wiener process on a
stochastic base $(\Omega, \mathscr{F}, \mathscr{F}_t, \mathbb{P})$
with covariance operator $Q = I$.
\end{assumption}
Formally, $W$ can be written (cf.
Da Prato and Zabczyk \cite{Daprato}) as the infinite sums
\begin{eqnarray*}
W_t=\sum\limits_{k\in\mathbb{N}}B_{k}(t)f_k,
\end{eqnarray*}
where $\{B_{k}(t)\}_{k\in \mathbb{N}}$ are mutually independent
real-valued Brownian motions on stochastic  base $(\Omega,
\mathscr{F}, \{\mathscr{F}_t\}_{t\geq0}, \mathbb{P})$, and
$\{f_k\}_{k\in\mathbb{N}}$ is any orthonormal basis on $U$.

We proceed with some further notation. Let $V$ be  another separable
Hilbert space with inner product $\langle\cdot, \cdot\rangle_V$. Let
$\mathscr{L}_2(U,V)$ denote the Hilbert space consisting of all
Hilbert-Schmidt operators from $U$ to $V$, where the inner product
is denoted by $\langle \cdot,\cdot\rangle_{\mathscr{L}_2(U,V)}$,
and the norm by $\|\cdot\|_{\mathscr{L}_2(U,V)}$.

\begin{assumption}[Operator G]
 Assume that $G:
\mathcal{H}^\alpha\rightarrow \mathscr{L}_2(U, \mathcal{H}^\alpha) $
satisfying $G(0)=0$, with $\alpha$ as in Assumption 2, is
Fr\'{e}chet differentiable up to order 2 and fulfills  the following
conditions: For one $r>0$ there exists a constant  $l_r>0$ such that
\begin{eqnarray}\label{g-condition-1}
\|G(u)\|_{\mathscr{L}_2(U, \mathcal{H}^\alpha)}\leq l_r\|u\|_\alpha,
\end{eqnarray}
\begin{eqnarray}\label{g-condition-2}
\|G'(u)\cdot v\|_{\mathscr{L}_2(U, \mathcal{H}^\alpha)}\leq l_r
\|v\|_\alpha
\end{eqnarray}
and
\begin{eqnarray}\label{g-condition-3}
\|G''(u)\cdot (v, w)\|_{\mathscr{L}_2(U, \mathcal{H}^\alpha)}\leq
l_r \|v\|_\alpha\|w\|_\alpha,
\end{eqnarray}
for all $u, v, w\in \mathcal{H}^\alpha$ with $ \|u\|_\alpha \leq r$,
where we use notations $G'(u)$ and $G''(u)$ denote the first and
second Fr\'{e}chet derivatives at point $u$, respectively.
\end{assumption}
Let us remark that the assumption on the second Frechet-derivative is only posed for simplicity of proofs
when we bound terms like $G(u)-G'(0)\cdot u$.

To give a meaning to problem \eqref{original-eqn}, we adapt the
concept of local mild solution as in \cite{Dirk-2014-JEE}.

\begin{definition}(Local mild solution)
An $\mathcal {H}^\alpha$-valued
process  $\{u(t)\}_{t\in[0, T]}$, is called a mild solution of
problem \eqref{original-eqn} if  for some stopping time $\tau_{\text{ex}}$ we
have on a set of probability $1$ that $\tau_{\text{ex}}>0$,  
$u \in
C^0([0,\tau_{\text{ex}}), \mathcal{H}^\alpha)$  and
\begin{eqnarray*}
u(t)=e^{t\mathcal{A}}u(0)+\int_0^t
e^{(t-s)\mathcal{A}}[\varepsilon^2
\mathcal{L}u(s)+\mathcal{F}(u(s))]ds+\varepsilon\int_0^t G(u(s))
dW(t)
\end{eqnarray*}
for all $t\in(0, \tau_{\text{ex}})$.
\end{definition}
The proof of the existence and the uniqueness   of a   local mild
solution is standard under our assumptions, and hence is omitted here.
For locally Lipschitz nonlinearities this follows using a cut-off
argument and Banach's fixed-point theorem.
For details see for example \cite{Daprato} or
\cite{Chow}.

Let us remark that one can choose $\tau_{\text{ex}}$ such that
with probability $1$ either $\tau_{\text{ex}}=\infty$
or $\lim_{t\nearrow\tau_{\text{ex}}}\|u(t)\|_\alpha =\infty$.

\section{Formal Derivation and the Main Result}
We consider the local mild solution $u$  on the slow time-scale $T=\varepsilon^2t$ 
and assume that it is small of order $\mathcal{O}(\varepsilon)$.
Let us split it into
\begin{eqnarray}\label{u-split}
u(t)=\varepsilon a(\varepsilon^2t)+\varepsilon \psi(\varepsilon^2t),
\end{eqnarray}
with $a\in \mathcal{N}$ and $\psi\in\mathcal{S}$.
By projecting and rescaling \eqref{original-eqn} to
the slow time scale, we obtain
\begin{eqnarray}
da(T)&=&\left[\mathcal{L}_ca(T)+\mathcal{F}_c\left(a(T)+\psi(T)\right)\right]dT \nonumber\\
&&\qquad +\frac{1}{\varepsilon} {G}_c\left(\varepsilon
a(T)+\varepsilon\psi(T)\right)d\tilde{W}(T) \label{a-equation}
\end{eqnarray}
and
\begin{eqnarray}
d\psi(T)&=&[\frac{1}{\varepsilon^2}\mathcal{A}_s\psi(T)+\mathcal{L}_s\psi(T)+\mathcal{F}_s(a(T)+\psi(T))]dT\nonumber\\
&&\qquad +\frac{1}{\varepsilon}G_s(\varepsilon
a(T)+\varepsilon\psi(T))d\tilde{W}(T),\label{psi-equation}
\end{eqnarray}
where $\tilde{W}(T):=\varepsilon W(\varepsilon^{-2}T)$ is a rescaled
version of the Wiener process. These equations can be
written in the   integral form using the mild formulation:
\begin{eqnarray}
a(T)&=&a(0)+\int_0^T\mathcal{L}_ca(\tau)d\tau+\int_0^T\mathcal{F}_c(a(\tau)+\psi(\tau))d\tau\nonumber\\
&&\qquad + \frac{1}{\varepsilon}\int_0^TG_c(\varepsilon
a(\tau)+\varepsilon\psi(\tau))d\tilde{W}_\tau
\label{a-integral-equation}
\end{eqnarray}
and
\begin{eqnarray}
\psi(T)&=&e^{\mathcal{A}_sT\varepsilon^{-2}}\psi(0)+
\int_0^Te^{\mathcal{A}_s(T-\tau)\varepsilon^{-2}}\mathcal{L}_s\psi(\tau)d\tau\nonumber\\
&&\qquad +\int_0^Te^{\mathcal{A}_s(T-\tau)\varepsilon^{-2}}\mathcal{F}_s(a(\tau)+\psi(\tau))d\tau\nonumber\\
&&\qquad +\frac{1}{\varepsilon}\int_0^Te^{\mathcal{A}_s(T-\tau)\varepsilon^{-2}}G_s(\varepsilon
a(\tau)+\varepsilon\psi(\tau))d\tilde{W}_\tau.
\label{psi-integral-equation-0}
\end{eqnarray}

We shall see later that $\psi$ is small as long as $a$ is of order one.
Thus by neglecting all $\psi$-dependent terms in \eqref{a-equation} or \eqref{a-integral-equation}
and expanding the diffusion
we obtain  the {\em amplitude equation}
\begin{eqnarray}
\label{amplitude-differential}
db(\tau)= \mathcal{L}_cb(\tau)d\tau+ \mathcal{F}_c(b(\tau))d\tau+
[G_c'(0)\cdot b(\tau)]\;d\tilde{W}_\tau,\quad b(0)=a(0).
\end{eqnarray}
This is equivalent to the integral equation
\begin{eqnarray}
 b(T)=a(0)+\int_0^T\mathcal{L}_cb(\tau)d\tau+\int_0^T\mathcal{F}_c(b(\tau))d\tau+\int_0^T
 [G_c'(0)\cdot b(\tau)]d\tilde{W}_\tau.
 \label{amplitude-integral}
\end{eqnarray}

With our main assumptions we have the following main result on the approximation
by amplitude equation, which is proved later at the end of Section
\ref{estimate}.

\begin{theorem} \label{Theorem}
Let the Assumptions 1 - 5 be satisfied and let $u$ be the local mild
solution of \eqref{original-eqn} with initial condition
\[u(0)=\varepsilon a(0)+\varepsilon\psi(0),\]
where $a(0)\in \mathcal
{N}$, $\psi(0)\in \mathcal {S}$ and $b$ is the solution of the amplitude equation 
\eqref{amplitude-differential} with $b(0)=a(0)$. 
Then for any $p>1$,
$T_0>0$ and  all small  $\kappa\in(0, \frac{1}{19}),$ there exists a
constant $C>0$ such that for 
$\|u(0)\|_\alpha\leq\varepsilon^{\kappa/3}$ 
 we have
\[
\mathbb{P}\Big( \sup\limits_{t\in [0, \varepsilon^{-2}
T_0]}\Big\|u(t)- \varepsilon b(\varepsilon^2t) -\varepsilon
Q(\varepsilon^2 t)\Big\|_\alpha >\varepsilon^{2-19\kappa}\Big)\leq
C\varepsilon^p,
\]
where
\[
Q(T)= e^{\mathcal{A}_sT\varepsilon^{-2}}\psi(0).
\]
\end{theorem}

Let us remark that the additional term $e^{\mathcal{A}_sT\varepsilon^{-2}}\psi(0)$
in the approximation is exponentially small after any
short time of order $\varepsilon$ by the stability of the semigroup on $\mathcal{S}$.
This is an attractivity result for the space $\mathcal{N}$ and allows for slightly bigger $\psi(0)$.

Note moreover that we did not optimize the factor in front of the $\kappa$. Both
the $19\kappa$ in the final error estimate and the $-\kappa/3$ 
in the bound on the initial condition are not optimal. We use $\kappa$ mainly for 
technical reasons and think of it as being very small.

Let us finally give some remarks on straightforward extensions of the result presented here.

\begin{remark}[Other nonlinear terms]
 We could add higher order terms to the SPDE like quartic or quintic, for example.
Formally, they are of higher order and  we do not expect to change the result very much.

Quadratic nonlinear terms $B(u,u)$ are quite different. Formally, 
we obtain in the amplitude equation the additional terms 
$\frac1\varepsilon B_c(a,a)$ and $\frac2\varepsilon B_c(a,\psi)$.
So either we need to change the scaling of the equation, consider smaller noise, 
and obtain an amplitude equation with quadratic nonlinearity,
or alternatively (as  $B_c(a,a)=0$ in many applications) we have to identify the mixed term $B(a,\psi)$.
Even if $\psi$ is small of order $\mathcal{O}(\varepsilon)$, then $B(a,\psi)$ is of order $\mathcal{O}(1)$ 
and we need to identify how $\psi$ depends on $a$.

See \cite{Dirk-2014-JEE} for a discussion in the case of additive noise.
\end{remark}

 \begin{remark}[Additive noise or quadratic diffusion]
The Assumption that $G(0)=0$ is crucial for our result, as for additive noise one sees already
in the formal calculation above, that we need a different scaling.
We expect to need $\varepsilon^2 G(u)dW$ in \eqref{original-eqn} which leads to an additive noise term $G_c(0)dW$ in the amplitude equation.
The proofs and the final theorem should nevertheless be very similar.

If we assume that not only $G(0)=0$ but also $G'(0)=0$, then we expect to have $G(u)dW$ in \eqref{original-eqn}
which leads to $[G_c''(0)\cdot(b,b)]dW$ in the amplitude equation.
Again the proofs  should be similar, but for the error estimate we might
need additional assumptions on the third derivative of $G$.
\end{remark}
\section{Estimates and Proof}\label{estimate}
Before proving the main results, we need to state some technical
lemmas used later in the proof. Also, we need to introduce a stopping time  in
connection with process $(a, \psi)$. This stopping time is equivalent to a cut-off
in \eqref{original-eqn} at radius $\varepsilon^{1-\kappa}$.
Also this stopping time is the reason, why we only need local solutions for the SPDE.
\begin{definition}\label{stopping time}
For the $\mathcal{N}\times\mathcal{S}$-valued stochastic process
$(a,\psi)$ satisfying system \eqref{a-integral-equation}-
\eqref{psi-integral-equation-0} we define, 
for some time $T_0>0$ 
and small exponent $\kappa\in (0,\frac{1}{12})$, 
the stopping time $\tau^*$ as
\begin{eqnarray*}
\tau^*:=T_0\wedge\inf\left\{T>0:\|a(T)\|_\alpha>\varepsilon^{-\kappa}
\quad or \quad\|\psi(T)\|_\alpha>\varepsilon^{-\kappa}\right\}.
\end{eqnarray*}
\end{definition}
Next, we will denote by $Q(T)$, $I(T)$, $J(T)$ and $K(T)$ the
corresponding four terms arising in the right hand side of
\eqref{psi-integral-equation-0}, respectively, that is
\begin{eqnarray}
\psi(T)=Q(T)+I(T)+J(T)+K(T).
\label{e:DB1}
\end{eqnarray}

\begin{lemma}\label{lemma-1}
 Let the Assumption 1 - Assumption 5 be satisfied. For any $p>0$ and  $\tau^*$
 from the definition \ref{stopping time}, there exists a constant
 $C >0$ such that
\begin{equation}
\mathbb{E}\sup\limits_{0\leq T\leq\tau^*} \|I(T)\|_\alpha^p \leq C
\varepsilon^{2p-\kappa p}  \label{I-bound}
\end{equation}
and
\begin{equation}
\mathbb{E}\sup\limits_{0\leq T\leq\tau^*} \|J(T)\|_\alpha^p \leq C
\varepsilon^{2p-3\kappa p}. \label{J-bound}
\end{equation}
\end{lemma}

\begin{proof}
By   \eqref{semigroup} and definition \ref{stopping time}, we first have for $I$
\begin{eqnarray}
&&\mathbb{E}\sup\limits_{0\leq T\leq\tau^*}\|I(T)\|_\alpha^p \nonumber \\
&&\leq \mathbb{E}\sup\limits_{0\leq
T\leq\tau^*}\left[\int_0^T\|e^{\mathcal{A}_s(T-\tau)\varepsilon^{-2}}\mathcal{L}_s\psi(\tau)\|_\alpha
d\tau\right]^p\nonumber\\
&&\leq C\varepsilon^{\frac{2\beta p}{m}}\mathbb{E}\sup\limits_{0\leq
T\leq\tau^*}\left[\int_0^Te^{-\varepsilon^{-2}\rho(T-\tau)}(T-\tau)^{-\frac{\beta}{m}}
\|\mathcal{L}_s\psi(\tau)\|_{\alpha-\beta} d\tau\right]^p\nonumber\\
&&\leq C\varepsilon^{\frac{2\beta p}{m}}\mathbb{E}\sup\limits_{0\leq
T\leq\tau^*}\left[\int_0^Te^{-\varepsilon^{-2}\rho(T-\tau)}(T-\tau)^{-\frac{\beta}{m}}
\| \psi(\tau)\|_{\alpha } d\tau\right]^p\nonumber\\
&&\leq C\varepsilon^{\frac{2\beta p}{m}}\sup\limits_{0\leq
T\leq\tau^*}
\left[\int_0^Te^{-\varepsilon^{-2}\rho(T-\tau)}(T-\tau)^{-\frac{\beta}{m}}
\varepsilon^{-\kappa} d\tau\right]^p\nonumber\\
&&\leq C\varepsilon^{2p-\kappa p}\sup\limits_{0\leq
T\leq\tau^*}\left[\int_0^{\varepsilon^{-2}\rho T}e^{-r}r^{-\frac{\beta}{m}}dr\right]^p\nonumber\\
&&\leq C\varepsilon^{2p-\kappa p},\nonumber
\end{eqnarray}
so that \eqref{I-bound} follows. In view of Assumption 3, Definition
\ref{stopping time} and \eqref{semigroup} we obtain for $J$,
\begin{eqnarray}
&&\mathbb{E}\sup\limits_{0\leq T\leq\tau^*}\|J
(T)\|_\alpha^p\nonumber \\
&&\leq \mathbb{E}\sup\limits_{0\leq
T\leq\tau^*}\left[\int_0^T\|e^{\mathcal{A}_s(T-\tau)\varepsilon^{-2}}\mathcal{F}_s(a(\tau)+\psi(\tau))\|_\alpha
d\tau\right]^p\nonumber\\
&&\leq  C\varepsilon^{\frac{2\beta
p}{m}}\mathbb{E}\sup\limits_{0\leq
T\leq\tau^*}\left[\int_0^Te^{-\varepsilon^{-2}\rho(T-\tau)}(T-\tau)^{-\frac{\beta}{m}}
\|\mathcal{F}_s(a(\tau)+\psi(\tau))\|_{\alpha-\beta} d\tau\right]^p\nonumber\\
&&\leq C\varepsilon^{\frac{2\beta p}{m}}\mathbb{E}\sup\limits_{0\leq
T\leq\tau^*}\left[\int_0^Te^{-\varepsilon^{-2}\rho(T-\tau)}(T-\tau)^{-\frac{\beta}{m}}
\| a(\tau)+\psi(\tau)\|^3_{\alpha } d\tau\right]^p\nonumber\\
&&\leq C\varepsilon^{\frac{2\beta p}{m}}\sup\limits_{0\leq
T\leq\tau^*}
\left[\int_0^Te^{-\varepsilon^{-2}\rho(T-\tau)}(T-\tau)^{-\frac{\beta}{m}}
\varepsilon^{-3\kappa} d\tau\right]^p\nonumber\\
&&\leq C\varepsilon^{2p-3\kappa p}\sup\limits_{0\leq
T\leq\tau^*}\left[\int_0^{\varepsilon^{-2}\rho T}e^{-r}r^{-\frac{\beta}{m}}dr\right]^p\nonumber\\
&&\leq C\varepsilon^{2p-3\kappa p}.\nonumber
\end{eqnarray}
The proof of Lemma \ref{lemma-1} is thus completed.
\end{proof}

While for $I$ and $J$ we immediately had uniform bounds in time,
for $K$ we first establish bounds in  $L^p([0,\tau^*],\mathcal{H}^\alpha)$. 

\begin{lemma}\label{lemma-2}
Assume the setting of Lemma \ref{lemma-1}. Then it holds for every
$p>0$ that
\begin{eqnarray}
\mathbb{E}\sup\limits_{0\leq T\leq\tau^*}\int_0^T\|K
(\tau)\|_\alpha^p d\tau \leq C_p\varepsilon^{p-\kappa
p}.\label{K-intetral-bound}
\end{eqnarray}
\end{lemma}
\begin{proof}
Throughout this proof let $\lambda_0$ be a
positive constant less than $\lambda_{n+1}$ but close to it. For any $p>0$, it holds
\begin{eqnarray}
&&\mathbb{E}\sup\limits_{0\leq T\leq\tau^*}\int_0^T\|K
(\tau)\|_\alpha^p
d\tau\nonumber\\
&&=\mathbb{E}\sup\limits_{0\leq
T\leq\tau^*}\int_0^T\|\frac{1}{\varepsilon}\int_0^\tau
e^{\mathcal{A}_s(\tau-r)\varepsilon^{-2}}G_s(\varepsilon
a(r)+\varepsilon\psi(r))d\tilde{W}_r\|_{\alpha}^pd\tau\nonumber\\
&&\leq \mathbb{E}\int_0^{T_0}\|1_{[0,
\tau^*]}(\tau)\frac{1}{\varepsilon}\int_0^\tau
e^{\mathcal{A}_s(\tau-r)\varepsilon^{-2}}G_s(\varepsilon
a(r)+\varepsilon\psi(r))d\tilde{W}_r\|_{\alpha}^pd\tau \nonumber\\
&&\leq\mathbb{E}\int_0^{T_0}\|\frac{1}{\varepsilon}\int_0^{\tau\wedge\tau^*}e^{\mathcal{A}_s(\tau-r)\varepsilon^{-2}}G_s(\varepsilon
a(r)+\varepsilon\psi(r))d\tilde{W}_r\|_{\alpha}^pd\tau\nonumber\\
&&=\mathbb{E}\int_0^{T_0}\|\frac{1}{\varepsilon}\int_0^{\tau}1_{[0,
\tau^*]}(r)e^{\mathcal{A}_s(\tau-r)\varepsilon^{-2}}G_s(\varepsilon
a(r)+\varepsilon\psi(r))d\tilde{W}_r\|_{\alpha}^pd\tau\nonumber\\
&&=\frac{1}{\varepsilon^p}\int_0^{T_0}e^{-\varepsilon^{-2}\lambda_0p
\tau}\mathbb{E}\| \int_0^{\tau}1_{[0,
\tau^*]}(r)e^{(\mathcal{A}_s+\lambda_0I)(\tau-r)\varepsilon^{-2}}e^{
\varepsilon^{-2}\lambda_0r} \nonumber \\
&& \qquad \qquad\qquad \qquad\qquad \qquad \qquad \qquad G_s(\varepsilon
a(r)+\varepsilon\psi(r))d\tilde{W}_r\|_{\alpha}^pd\tau.\nonumber
\end{eqnarray}
By an application of the maximal inequality for stochastic convolutions
\cite{Hausenblas} based on the Riesz-Nagy theorem
(as $\mathcal{A}_s+\lambda_0$ generates a contraction semigroup on $\mathcal{S}$),
the condition \eqref{g-condition-1} for $G$,
and the definition of $\tau^*$, we obtain
\begin{eqnarray}
&&\mathbb{E}\sup\limits_{0 \leq T\leq\tau^*}
\int_0^T\|\frac{1}{\varepsilon}\int_0^\tau
e^{\mathcal{A}_s(\tau-r)\varepsilon^{-2}}G_s(\varepsilon
a(r)+\varepsilon\psi(r))d\tilde{W}_r\|_{\alpha}^pd\tau\nonumber\\
&&\leq C\int_0^{T_0}e^{-\varepsilon^{-2}\lambda_0p
\tau}\mathbb{E}[\int_0^\tau e^{ 2\varepsilon^{-2}\lambda_0r }1_{[0,
\tau^*]}(r) \|a(r)+\psi(r)\|^2_\alpha  dr]^{\frac{p}{2}}d\tau\nonumber\\
&&\leq C\int_0^{T_0}e^{-\varepsilon^{-2}\lambda_0p
\tau}[\int_0^\tau(\varepsilon^{-2\kappa}+1) e^{
2\varepsilon^{-2}\lambda_0
r}dr]^{\frac{p}{2}}d\tau\nonumber\\
&&\leq C\varepsilon^{p-\kappa
p}\int_0^{T_0}e^{-\varepsilon^{-2}\lambda_0p  \tau}[e^{
2\varepsilon^{-2}\lambda_0 \tau}-1]^{\frac{p}{2}}d\tau\nonumber\\
&&\leq C\varepsilon^{(1-\kappa)p+2}
\nonumber
\end{eqnarray}
where the constant may change from line to line, but it mainly depends on $p$, $T_0$, the bound on $G$, and $\lambda_0$.
\end{proof}
So we have seen in the previous lemmas that $\psi$ equals $Q$ plus a small term.
Next, let us rewrite the equation \eqref{a-equation} for $a$ as the amplitude
equation plus an error term (or residual).
\[
a(T)=a(0)+\int_0^T \left[\mathcal{L}_ca(\tau)+\mathcal{F}_c(a(\tau)) \right] d\tau
+\int_0^TG_c'(0) \cdot a(\tau)d\tilde{W}_\tau+R(T),
\]
where the error term is given by
\begin{eqnarray}
R (T)&=&\int_0^T \left[3\mathcal{F}_c(a(\tau),\psi(\tau),\psi(\tau))+3\mathcal{F}_c(a(\tau),a(\tau),\psi(\tau))+ \mathcal{F}_c(
\psi(\tau))\right]d\tau \nonumber\\
&&+\int_0^T[\frac{1}{\varepsilon}G_c(\varepsilon
a(\tau)+\varepsilon\psi(\tau))-G_c'(0)\cdot a(\tau)]d\tilde{W}_\tau.
\label{Remainder}
\end{eqnarray}
Let us now start to show that $R$ is small.
\begin{lemma}\label{lemma3}
Assume the setting of Lemma \ref{lemma-1}. For any $p>0$, there
exists a constant
 $C_p>0$ such that
\[
\mathbb{E}\sup\limits_{0\leq
T\leq\tau^*}\|\int_0^T\mathcal{F}_c(a(\tau),\psi(\tau),\psi(\tau))d\tau\|_\alpha^p\leq C_p\left(\varepsilon^{2p-7\kappa p}+\|\psi(0)\|^{2p}_\alpha
\varepsilon^{2p-\kappa p} \right). 
\]
\begin{proof}
It is direct to see that by brute force expansion of the cubic using $\psi=Q+I+K+J$ from \eqref{e:DB1} that
\begin{eqnarray}
&&\int_0^T\mathcal{F}_c(a(\tau),\psi(\tau),\psi(\tau))d\tau\nonumber\\
&&=\int_0^T\mathcal{F}_c(a(\tau),Q (\tau),Q (\tau))d\tau+\int_0^T\mathcal{F}_c(a(\tau),I (\tau),I (\tau))d\tau\nonumber\\
&&+\int_0^T\mathcal{F}_c(a(\tau),J (\tau),J (\tau))d\tau+\int_0^T\mathcal{F}_c(a(\tau),K (\tau),K (\tau))d\tau\nonumber\\
&&+2\int_0^T\mathcal{F}_c(a(\tau),Q (\tau),I (\tau))d\tau+2\int_0^T\mathcal{F}_c(a(\tau),Q (\tau),J (\tau))d\tau\nonumber\\
&&+2\int_0^T\mathcal{F}_c(a(\tau),Q (\tau),K (\tau))d\tau+2\int_0^T\mathcal{F}_c(a(\tau),I (\tau),J (\tau))d\tau\nonumber\\
&&+2\int_0^T\mathcal{F}_c(a(\tau),I(\tau),K (\tau))d\tau+2\int_0^T\mathcal{F}_c(a(\tau),J (\tau),K (\tau))d\tau\nonumber\\
&&:=\sum\limits_{k=1}^{10}R ^{1,k}(T).\label{lemma3-1}
\end{eqnarray}
We will estimate each term separately, which will all be very similar, as $I$, $J$, and $K$ are small. 
Only for $Q$ we need an additional averaging argument. First since all
$\mathcal{H}^\alpha$- norms are equivalent on $\mathcal{N}$, we get
\begin{eqnarray}
\mathbb{E}\sup\limits_{0\leq
T\leq\tau^*}\|R^{1,1}(T)\|_\alpha^p&\leq&
C\mathbb{E}\sup\limits_{0\leq
T\leq\tau^*}\|R^{1,1}(T)\|_{\alpha-\beta}^p\nonumber\\
&\leq&C\mathbb{E}\sup\limits_{0\leq
T\leq\tau^*}\left[\int_0^T\|\mathcal{F}_c(a(\tau),Q (\tau),Q (\tau))\|_{\alpha-\beta}d\tau\right]^p\nonumber\\
&\leq&C\mathbb{E}\sup\limits_{0\leq T\leq\tau^*}\left[\int_0^T
\|a(\tau)\|_\alpha\|Q (\tau)\|_\alpha^2d\tau\right]^p\nonumber\\
&\leq&C\varepsilon^{-\kappa
p}\left[\int_0^{T_0}\|e^{\mathcal{A}_s\tau\varepsilon^{-2}}\psi(0)\|^2_\alpha
d\tau\right]^p\nonumber\\
&\leq& C\varepsilon^{2p-\kappa p}\|\psi(0)\|^{2p}_\alpha.\nonumber
\end{eqnarray}
For $R ^{1,2}(T)$, we have
\begin{eqnarray}
\mathbb{E}\sup\limits_{0\leq T\leq\tau^*}\|R
^{1,2}(T)\|_\alpha^p&\leq& C\mathbb{E}\sup\limits_{0\leq
T\leq\tau^*}\|R_\varepsilon^{1,2}(T)\|_{\alpha-\beta}^p\nonumber\\
&\leq&C\mathbb{E}\sup\limits_{0\leq
T\leq\tau^*}\left[\int_0^T\|\mathcal{F}_c(a(\tau),I (\tau),I (\tau))\|_{\alpha-\beta}d\tau\right]^p\nonumber\\
&\leq&C\mathbb{E}\sup\limits_{0\leq T\leq\tau^*}\left[\int_0^T\|
a(\tau)\|_\alpha \|I (\tau)\|_{\alpha}^2d\tau\right]^p\nonumber
\end{eqnarray}
Due to definition   \ref{stopping time} and \eqref{I-bound}, we get
\begin{eqnarray}
\mathbb{E}\sup\limits_{0\leq T\leq\tau^*}\|R
^{1,2}(T)\|_\alpha^p&\leq& C \varepsilon^{-\kappa
p}\mathbb{E}\sup\limits_{0\leq
T\leq\tau^*}\int_0^T\|I (\tau)\|_{\alpha}^{2p}d\tau\nonumber\\
&\leq& C \varepsilon^{4p-3\kappa p} .\nonumber
\end{eqnarray}
By proceeding with analogous arguments, we can show the following
results for all other terms:
\begin{eqnarray}
\mathbb{E}\sup\limits_{0\leq T\leq\tau^*}\|R
^{1,3}(T)\|_\alpha^p&\leq& C_p\varepsilon^{4p-7\kappa p},\nonumber
\end{eqnarray}

\begin{eqnarray}
\mathbb{E}\sup\limits_{0\leq T\leq\tau^*}\|R
^{1,4}(T)\|_\alpha^p&\leq& C_p\varepsilon^{2p-3\kappa p},\nonumber
\end{eqnarray}

\begin{eqnarray}
\mathbb{E}\sup\limits_{0\leq T\leq\tau^*}\|R
^{1,5}(T)\|_\alpha^p&\leq& C_p\varepsilon^{- \kappa p}
(\varepsilon^{2p}\|\psi(0)\|_\alpha^{2p}+\varepsilon^{4p-2\kappa
p}),\nonumber
\end{eqnarray}

\begin{eqnarray}
\mathbb{E}\sup\limits_{0\leq T\leq\tau^*}\|R
^{1,6}(T)\|_\alpha^p&\leq& C_p\varepsilon^{- \kappa p}
(\varepsilon^{2p}\|\psi(0)\|_\alpha^{2p}+\varepsilon^{4p-6\kappa
p}),\nonumber
\end{eqnarray}

\begin{eqnarray}
\mathbb{E}\sup\limits_{0\leq T\leq\tau^*}\|R
^{1,7}(T)\|_\alpha^p&\leq&  C_p\varepsilon^{- \kappa p}
(\varepsilon^{2p}\|\psi(0)\|_\alpha^{2p}+\varepsilon^{2p-2\kappa
p}),\nonumber
\end{eqnarray}

\begin{eqnarray}
\mathbb{E}\sup\limits_{0\leq T\leq\tau^*}\|R
^{1,8}(T)\|_\alpha^p&\leq&  C_p\varepsilon^{- \kappa p}
(\varepsilon^{4p-2\kappa p}+\varepsilon^{4p-6\kappa p}),\nonumber
\end{eqnarray}

\begin{eqnarray}
\mathbb{E}\sup\limits_{0\leq T\leq\tau^*}\|R
^{1,9}(T)\|_\alpha^p&\leq&  C_p\varepsilon^{- \kappa p}
(\varepsilon^{4p-2\kappa p}+\varepsilon^{2p-2\kappa p}),\nonumber
\end{eqnarray}
and
\begin{eqnarray}
\mathbb{E}\sup\limits_{0\leq T\leq\tau^*}\|R
^{1,10}(T)\|_\alpha^p&\leq&   C_p\varepsilon^{- \kappa p}
(\varepsilon^{4p-6\kappa p}+\varepsilon^{2p-2\kappa p}).\nonumber
\end{eqnarray}
Collecting all estimates for terms appearing in \eqref{lemma3-1} we
finish the proof.
\end{proof}
\end{lemma}

By the same arguments which we used to derive Lemma \ref{lemma3}, we are able to
achieve following results:
\begin{lemma}\label{lemma4}
Assume the setting of Lemma \ref{lemma-1}. For any $p>0$, there
exists a constant $C >0$ such that
\begin{eqnarray*}
\mathbb{E}\sup\limits_{0\leq
T\leq\tau^*}\|\int_0^T\mathcal{F}_c(a(\tau),a(\tau),\psi(\tau))d\tau\|_\alpha^p\leq
C_p\left(\varepsilon^{ p-5\kappa
p}+\|\psi(0)\|_\alpha^{p}\varepsilon^{ 2p-2\kappa
p}\right).\label{lemma4-bound}
\end{eqnarray*}
\end{lemma}

\begin{lemma}\label{lemma5}
Assume the setting of Lemma \ref{lemma-1}. For any $p>0$, there
exists a constant $C >0$ such that
\begin{eqnarray}
\mathbb{E}\sup\limits_{0\leq T\leq\tau^*}\|\int_0^T\mathcal{F}_c(
\psi(\tau))d\tau\|_\alpha^p\leq C_p\left(\varepsilon^{3p-9\kappa
p}+\|\psi(0)\|_\alpha^{3p}\varepsilon^{2p}\right).\label{lemma5-bound}
\end{eqnarray}
\begin{proof}

As we noticed before,  all norms in  finite dimensional space
$\mathcal{N}$ are equivalent. Thanks to \eqref{F-condition-1}, we
get
\begin{eqnarray*}
\|\mathcal{F}_c(\psi(\tau))\|_\alpha&\leq& C\left(
\|Q(\tau)+I(\tau)+J(\tau)+K(\tau)\|_\alpha^3\right)\\
&\leq&
C\left(\|Q(\tau)\|_\alpha^3+\|I(\tau)\|_\alpha^3+\|J(\tau)\|_\alpha^3+\|K(\tau)\|_\alpha^3\right).
\end{eqnarray*}
Thus, according   to the H\"{o}lder inequality, this implies
\begin{eqnarray*}
&&\mathbb{E}\sup\limits_{0\leq T\leq\tau^*}\|\int_0^T\mathcal{F}_c(
\psi(\tau))d\tau\|_\alpha^p\\
&&\leq C_p\mathbb{E}\sup\limits_{0\leq
T\leq\tau^*}\left[\int_0^T\|Q(\tau)\|_\alpha^3d\tau\right]^p\\
&&\quad+C_p\mathbb{E}\sup\limits_{0\leq
T\leq\tau^*}\left\{\int_0^T\big[\|I(\tau)\|_\alpha^{3p}+\|J(\tau)\|_\alpha^{3p}+\|K(\tau)\|_\alpha^{3p}\big]d\tau\right\}.
\end{eqnarray*}
It is easy to check that the first term appearing in the right side
of above inequality is bounded by
$C_p\varepsilon^{2p}\|\psi(0)\|_\alpha^{3p}$ for a constant $C_p>0$.
Due to Lemma \ref{lemma-1} and Lemma \ref{lemma-2} we can conclude
that the second term is bounded by $C_p\varepsilon^{3p-9\kappa p}$.
Therefore, we finish the proof and obtain \eqref{lemma5-bound}.
\end{proof}
\end{lemma}

\begin{lemma}\label{lemma6}
Assume the setting of Lemma \ref{lemma-1}. For any $p>0$, there
exists a constant $C_p>0$ such that
\begin{eqnarray}
\mathbb{E}\sup\limits_{0\leq T\leq\tau^*}\|
\int_0^T[\frac{1}{\varepsilon}G_c(\varepsilon
a(\tau)+\varepsilon\psi(\tau))-G_c'(0)\cdot
a(\tau)]d\tilde{W}_\tau\|_\alpha^p\leq C_p\varepsilon^{p-3\kappa
p}.\nonumber\label{lemma6-bound}
\end{eqnarray}
\begin{proof}
Using Burkholder-Davis-Gundy inequality, we have
\begin{eqnarray}
&&\!\!\!\!\!\!\!\!\!\!\!\!\!\!\!\mathbb{E}\sup\limits_{0\leq
T\leq\tau^*}\| \int_0^T[\frac{1}{\varepsilon}G_c(\varepsilon
a(\tau)+\varepsilon\psi(\tau))-G_c'(0)\cdot
a(\tau)]d\tilde{W}_\tau\|_\alpha^p\nonumber \\
&&\!\!\!\!\!\!\!\!\!\!\!\!\!\!\!\leq
C\mathbb{E}\left[\int_0^{T_0}1_{[0,\tau*]}(\tau)\|\frac{1}{\varepsilon}G_c(\varepsilon
a(\tau)+\varepsilon\psi(\tau))-G_c'(0)\cdot
a(\tau)\|^2_{\mathscr{L}_2(U, \mathcal{H}^\alpha)}d\tau\right]^{\frac{p}{2}}\nonumber\\
&&\!\!\!\!\!\!\!\!\!\!\!\!\!\!\!\leq
C\mathbb{E}\left[\int_0^{T_0}1_{[0,\tau*]}(\tau)\|\frac{1}{\varepsilon}
G(\varepsilon a(\tau)+\varepsilon\psi(\tau))- G'(0)\cdot
a(\tau)\|^2_{\mathscr{L}_2(U,
\mathcal{H}^\alpha)}d\tau\right]^{\frac{p}{2}}. \label{Lemma2.6-1}
\end{eqnarray}
By using the Taylor formula, we obtain
\begin{eqnarray}
&&\frac{1}{\varepsilon}  G(\varepsilon
a(\tau)+\varepsilon\psi(\tau))-
G'(0)\cdot a(\tau)\nonumber\\
&&=\frac{1}{\varepsilon}[G(0)+G'(0)\cdot(\varepsilon
a(\tau)+\varepsilon\psi(\tau))\nonumber
\\
&&+\frac{1}{2}G''(z(\tau))\cdot(\varepsilon
a(\tau)+\varepsilon\psi(\tau),\varepsilon
a(\tau)+\varepsilon\psi(\tau))] -G'(0)\cdot a(\tau)\nonumber\\
 &&= G'(0)\cdot \psi(\tau)+\frac{\varepsilon}{2}G''(z(\tau))\cdot (
a(\tau)+ \psi(\tau),  a(\tau)+ \psi(\tau)),\nonumber
\end{eqnarray}
where $z(\tau)$ is a vector on the line segment connecting $0$ and
$\varepsilon a(\tau)+\varepsilon\psi(\tau)$. Now, as a consequence
of the condition \eqref{g-condition-3}, we have
\begin{eqnarray}
\lefteqn{\|\frac{1}{\varepsilon}  G(\varepsilon
a(\tau)+\varepsilon\psi(\tau))-
G'(0)\cdot a(\tau)\|_{L_2^0}^2}\nonumber\\
&& \qquad \qquad \leq C\|\psi\|_\alpha^2
+C\varepsilon^2\|a(\tau)\|_\alpha^4+C\varepsilon^2\|\psi(\tau)\|_\alpha^4.\nonumber
\end{eqnarray}
Therefore, if we plug the estimate above   into \eqref{Lemma2.6-1},
we get
\begin{eqnarray}
&&\mathbb{E}\sup\limits_{0\leq T\leq\tau^*}\|
\int_0^T[\frac{1}{\varepsilon}G_c(\varepsilon
a(\tau)+\varepsilon\psi(\tau))-G_c'(0)\cdot
a(\tau)]d\tilde{W}_\tau\|_\alpha^p\nonumber \\
&&\leq
C\mathbb{E}\left[\int_0^{T_0}1_{[0,\tau*]}(\tau)(\|\psi(\tau)\|^2_\alpha+\varepsilon^2\|a(\tau)\|_\alpha^4
+\varepsilon^2\|\psi(\tau)\|_\alpha^4)d\tau\right]^{\frac{p}{2}}\nonumber\\
&&\leq C_p\varepsilon^{p-2\kappa
p}+C_p\mathbb{E}\left[\int_0^{T_0}1_{[0,\tau*]}(\tau)\|\psi(\tau)\|^2_\alpha
d\tau\right]^{\frac{p}{2}},\nonumber
\end{eqnarray}
where   the last estimate following from the definition of $\tau^*$.
From the expression for $\psi(\tau)$ and H\"{o}lder's inequality, we
get
\begin{eqnarray}
\lefteqn{\mathbb{E}\left[\int_0^{T_0}1_{[0,\tau*]}(\tau)\|\psi(\tau)\|^2_\alpha
d\tau\right]^{\frac{p}{2}}
} 
\nonumber \\
&&\leq C_p\mathbb{E}\left[\int_0^{T_0}1_{[0,\tau*]}(\tau)\big(\|Q(\tau)\|^2_\alpha
+\|I(\tau)\|^2_\alpha+\|J(\tau)\|^2_\alpha
+\|K(\tau)\|^2_\alpha\big)d\tau\right]^{\frac{p}{2}}
\nonumber\\
&&\leq C_p \mathbb{E} \Big[ \varepsilon^p\|\psi(0)\|_\alpha^p+
\sup\limits_{0\leq\tau\leq\tau^*}\|I(\tau)\|^p_\alpha+
 \sup\limits_{0\leq\tau\leq\tau^*}\|J(\tau)\|^p_\alpha
+  \int_0^{\tau^*}\!\|K(\tau)\|_\alpha^p d\tau \Big].
\nonumber
\end{eqnarray}
Recalling Lemma \ref{lemma-1} and Lemma \ref{lemma-2}, we thus have
\[
\mathbb{E}\sup\limits_{0\leq T\leq\tau^*}\|
\int_0^T[\frac{1}{\varepsilon}G_c(\varepsilon
a(\tau)+\varepsilon\psi(\tau))-G_c'(0)\cdot
a(\tau)]d\tilde{W}_\tau\|_\alpha^p
\leq C_p\varepsilon^{p-3\kappa p}.
\]

\end{proof}
\end{lemma}

Due to  Lemmas \ref{lemma3}, \ref{lemma4}, \ref{lemma5},
and \ref{lemma6}, we readily obtain the following estimate for
the remainder $R$ defined in \eqref{Remainder}.
\begin{lemma}\label{lemma-7}
Assume the setting of  Lemma \ref{lemma-1} and suppose furthermore
that $\|\psi(0)\|_\alpha\leq\varepsilon^{-\frac{1}{3}\kappa}$. 
Then for
any $p>0$, there exists a constant $C_p>0$ such that
\begin{eqnarray}
\mathbb{E}\sup\limits_{0\leq
T\leq\tau^*}\|R(T)\|_\alpha^p&\leq&C_p\varepsilon^{p-9\kappa
p}.\label{lemma-7-bound}
\end{eqnarray}
\end{lemma}

 In what follows, we shall consider the amplitude equation \eqref{amplitude-differential}
associated with \eqref{a-equation} and
we show the following uniform bound on its solution $b$.
This is crucial in order to remove the stopping time from the error estimate.
Moreover note, that our assumptions do not imply global solutions for the SPDE,
we rely on the existence of global solutions for the amplitude equation, 
which is also ensured by the following Lemma.

\begin{lemma}\label{lemma-8}
Let the Assumptions 1 - 5 be satisfied. For any $p>1$,
there exists a constant $C_p>0$ such that
\begin{eqnarray}
\mathbb{E}\sup\limits_{0\leq T\leq
T_0}\|b(T)\|_\alpha^p&\leq&C_p\|a(0)\|^p_\alpha
\label{lemma-8-bound}
\end{eqnarray}
\end{lemma}

\begin{proof}
This proof is relatively straightforward using It\^o-formula for powers of the norm.
For large $p>2$ define the twice continuously differentiable function
\begin{eqnarray}
f(\cdot) =\|\cdot \|^p:\;  \mathcal{H}\rightarrow\mathbb{R}.
\label{f-aux}
\end{eqnarray}
Directly, for any $x, h\in \mathcal{H}$ we have
\begin{eqnarray*}
f'(x)h=p\|x\|^{p-2}\langle x, h\rangle
\end{eqnarray*}
and
\begin{eqnarray}
f''(x)(h,h)&=&p(p-2)\|x\|^{p-4}\langle x, h\rangle\langle x,
h\rangle+p\|x\|^{p-2}\langle h,h\rangle\nonumber\\
&\leq& p(p-1)\|x\|^{p-2}\|h\|^2, 
\end{eqnarray}
so that
\begin{eqnarray}\label{derivative-1}
\lefteqn{\text{trace}[f''(b(\tau))G_c'(0)b(\tau))(G_c'(0)b(\tau))^*]}\nonumber\\
&&\leq p(p-1)\|b(\tau)\|^{p-2}\text{trace}[(G_c'(0)b(\tau))(G_c'(0)b(\tau))^*]\nonumber \\
&&\leq C p(p-1)\|b(\tau)\|^{p}.\label{Tr_1}
\end{eqnarray}
Applying  It\^{o}'s formula  \cite[Theorem 2.9]{Gawarecki} and
\eqref{Tr_1} we obtain that
\begin{eqnarray}
\|b(T)\|^p&\leq&\|a(0)\|^p+p\int_0^T\|b(\tau)\|^{p-2}\langle\mathcal{L}_cb(\tau),
b(\tau)\rangle
d\tau\nonumber\\
&&+p\int_0^T\|b(\tau)\|^{p-2}\langle\mathcal{F}_c(b(\tau)),
b(\tau)\rangle d\tau\nonumber\\
&&+p\int_0^T\|b(\tau)\|^{p-2}\langle b(\tau),G_c'(0)\cdot b(\tau)d\tilde{W}_\tau\rangle\nonumber\\
&&+\frac{1}{2}C p(p-1)\int_0^T \|b(\tau)\|^{p}d\tau\nonumber.
\end{eqnarray}
Therefore, using {Assumption 2} and the bound on $F$ from  \eqref{F-condition-2}, we get
\begin{eqnarray}
\|b(T)\|^p&\leq&\|a(0)\|^p +p\int_0^T\|b(\tau)\|^{p-2}\langle
b(\tau),G_c'(0)\cdot b(\tau)d\tilde{W}_\tau\rangle\nonumber\\
&&+C_p\int_0^T \|b(\tau)\|^{p}d\tau\label{lemma-8-1}.
\end{eqnarray}
For any stopping time $\mathcal{T}\leq T_0$, by Burkholder-Davis-Gundy inequality,
we obtain
\begin{eqnarray*}
&&p\mathbb{E}\sup\limits_{0\leq T\leq
\mathcal{T}}\int_0^T\|b(\tau)\|^{p-2}\langle
b(\tau),G_c'(0)\cdot b(\tau)d\tilde{W}_\tau\rangle\\
&&\leq
3p\mathbb{E}\left[\int_0^\mathcal{T}\sum\limits_{k=1}^\infty\|b(\tau)\|^{2p-4}\langle
b(\tau),G_c'(0)\cdot b(\tau)e_k\rangle^2
d\tau\right]^{\frac{1}{2}}\\
&&\leq
C\mathbb{E}\left[\int_0^\mathcal{T}\|b(\tau)\|^{2p}d\tau\right]^\frac{1}{2}\\
&&\leq C\mathbb{E}\left[\sup\limits_{0\leq T\leq
\mathcal{T}}\|b(\tau)\|^p\int_0^\mathcal{T}\|b(\tau)\|^{p}d\tau\right]^\frac{1}{2}\\
&&\leq \frac{1}{2}\mathbb{E}\sup\limits_{0\leq T\leq
\mathcal{T}}\|b(\tau)\|^p+C_p\int_0^\mathcal{T}\mathbb{E}\sup\limits_{0\leq
s\leq\tau}\|b(s)\|^{p}d\tau,
\end{eqnarray*}
where we applied Young's inequality in the final step.
Therefore, using  \eqref{lemma-8-1}, we obtain
\begin{eqnarray*}
\mathbb{E}\sup\limits_{0\leq T\leq \mathcal{T}}\|b(T)\|^p&\leq&
2\|a(0)\|^p+C_p\int_0^\mathcal{T}\mathbb{E}\sup\limits_{0\leq
s\leq\tau}\|b(\tau)\|^{p}d\tau.
\end{eqnarray*}
Note that we need to use a stopping time $\mathcal{T}$  here, as 
initially, we do not know that the moments of $b$ are finite.
Thus we consider only the $\mathcal{T}$'s which ensures this.

As the equation above holds for any stopping time,  we derive by using Gronwall's lemma
\[ 
\mathbb{E}\sup\limits_{0\leq T\leq T_0}\|b(\tau)\|^p
\leq
C_p\|a(0)\|^p.
\]
This finishes the proof.
\end{proof}
The next step now is to remove the error from the equation for $a$ to obtain the amplitude equation.
We show an error estimate between $a$ and the solution $b$ of the amplitude equation.
\begin{lemma}\label{lemma-9}
Assume the setting of Lemma \ref{lemma-1} and suppose furthermore
that $\|\psi(0)\|_\alpha\leq \varepsilon^{-\frac{1}{3}\kappa}$. 
For
any $p>0$, there exists a constant $C_p>1$ such that
\begin{eqnarray}
\mathbb{E}\sup\limits_{0\leq T\leq \tau^*}\|a(T)-b(T)\|^p\leq
C_p\varepsilon^{p-18\kappa p}.
\end{eqnarray}
\end{lemma}
\begin{proof}
For the proof we derive an equation for the error $a-b$ and proceed similarly than for the bound on $b$.
But as $R$ (defined in \eqref{Remainder}) is not differentiable in the Ito-sense,
we first substitute $\varphi:=a-R$. Clearly, we have
\begin{eqnarray*}
\varphi(T)&=&a(0)+\int_0^T\mathcal{L}_c(\varphi(\tau)+R(\tau))d\tau+\int_0^T\mathcal{F}_c(\varphi(\tau)+R(\tau))d\tau\\
&&+\int_0^TG_c'(0)\cdot(\varphi(\tau)+R(\tau))d\tilde{W}_\tau.
\end{eqnarray*}
Defining the error  $h:=b-\varphi=b-a+R$,
we get
\begin{eqnarray}
h(T)&=&\int_0^T\mathcal{L}_ch(\tau)d\tau-\int_0^T\mathcal{L}_cR(\tau)d\tau+\int_0^T\mathcal{F}_c(b(\tau))d\tau\nonumber\\
&&-\int_0^T\mathcal{F}_c(b(\tau)-h(\tau)+R(\tau))d\tau+\int_0^TG_c'(0)(h(\tau)-R(\tau))d\tilde{W}_\tau.\nonumber
\end{eqnarray}
Let  $f$ be the $p$-th power of the norm as in \eqref{f-aux}. 
By using again \eqref{derivative-1} we have
\begin{eqnarray*}
\lefteqn{\text{trace}[f''(h(\tau))(G_c'(0)(b(\tau)-R(\tau)))(G_c'(0)(b(\tau)-R(\tau)))^*]}\\
&&\leq C p(p-1)\|b(\tau)\|^{p-2}\|h(\tau)-R(\tau)\|^2.
\end{eqnarray*}
Applying  It\^{o}'s formula and using the estimate above, we obtain
\begin{eqnarray}
\|h(T)\|^p
&\leq&p\int_0^T\|h(\tau)\|^{p-2}\langle\mathcal{L}_ch(\tau),h(\tau)\rangle d\tau
\nonumber \\ 
&& -p\int_0^T\|h(\tau)\|^{p-2}\langle\mathcal{L}_cR(\tau),h(\tau)\rangle
d\tau\nonumber\\
&&+p\int_0^T\|h(\tau)\|^{p-2}\langle\mathcal{F}_c(b(\tau))-\mathcal{F}_c(b(\tau)-h(\tau)+R(\tau)),h(\tau)\rangle
d\tau\nonumber\\
&&+p\int_0^T\|h(\tau)\|^{p-2}\langle
h(\tau),G_c'(0)\cdot(h(\tau)-R(\tau))d\tilde{W}_\tau\rangle\nonumber\\
&&+\frac{1}{2}C p(p-1)\int_0^T\|h(\tau)\|^{p-2}
\|h(\tau)-R(\tau)\|^2 d\tau. \nonumber
\end{eqnarray}
By condition \eqref{F-condition-4} and Cauchy-Schwarz inequality, we
derive
\begin{eqnarray}
\|h(T)\|^p&\leq&C_p\int_0^T\|h(\tau)\|^pd\tau+C_p\int_0^T\|h(\tau)\|^{p-1}\|R(\tau)\|d\tau\nonumber\\
&&+C_p\int_0^T\|h(\tau)\|^{p-2}\|R(\tau)\|^2d\tau+C_p\int_0^T\|h(\tau)\|^{p-2}\|R(\tau)\|^4d\tau\nonumber\\
&&+C_p\int_0^T\|h(\tau)\|^{p-2}\|R(\tau)\|^2\|b(\tau)\|^2d\tau\nonumber\\
&&+p\int_0^T\|h(\tau)\|^{p-2}\langle
h(\tau),G_c'(0)\cdot(h(\tau)-R(\tau))d\tilde{W}_\tau\rangle.\nonumber
\end{eqnarray}
Then, Young's inequality yields
\begin{eqnarray}
\|h(T)\|^p&\leq&C_p\int_0^T\|h(\tau)\|^pd\tau+C_p\int_0^T\|R(\tau)\|^pd\tau+C_p\int_0^T \|R(\tau)\|^{2p}d\tau \nonumber\\
&&+C_p\int_0^T\|R(\tau)\|^p\|b(\tau)\|^pd\tau\nonumber\\
&&+p\int_0^T\|h(\tau)\|^{p-2}\langle
h(\tau),G_c'(0)[h(\tau)-R(\tau)]d\tilde{W}_\tau\rangle.\label{lemma-9-1}
\end{eqnarray}
The last term on the right hand side of \eqref{lemma-9-1} is bounded
as follows. By Burkholder-Davis-Gundy inequality for any
$\mathcal{T}\in [0, T_0]$, we get
\begin{eqnarray}
&&\mathbb{E}\sup\limits_{0\leq
T\leq\tau^*\wedge\mathcal{T}}\left|p\int_0^T\|h(\tau)\|^{p-2}\langle
h(\tau),G_c'(0)[h(\tau)-R(\tau)]d\tilde{W}_\tau\rangle\right|\nonumber\\
&&\leq
3p\mathbb{E}\left[\int_0^{\tau^*\wedge\mathcal{T}}\|h(\tau)\|^{2p-4}\|h(\tau)\|^2\|h(\tau)-R(\tau)\|^2d\tau\right]^{\frac{1}{2}}\nonumber\\
&&\leq
C_p\mathbb{E}\left[\int_0^{\tau^*\wedge\mathcal{T}}\big[\|h(\tau)\|^{2p}+\|h(\tau)\|^{2p-2}\|R(\tau)\|^2\big]d\tau\right]^{\frac{1}{2}}.\nonumber
\end{eqnarray}
Using again  Young's inequality implies
\begin{eqnarray}
&&\mathbb{E}\sup\limits_{0\leq
T\leq\tau^*\wedge\mathcal{T}}\left|p\int_0^T\|h(\tau)\|^{p-2}\langle
h(\tau),G_c'(0)[h(\tau)-R(\tau)]d\tilde{W}_\tau\rangle\right|\nonumber\\
&&\leq
C_p\mathbb{E}\left[\int_0^{\tau^*\wedge\mathcal{T}}\|h(\tau)\|^{2p}d\tau\right]^\frac{1}{2}
+C_p\mathbb{E}\left[\int_0^{\tau^*\wedge\mathcal{T}}\|R(\tau)\|^{2p}d\tau\right]^\frac{1}{2}\nonumber\\
&&\leq\frac{1}{2}\mathbb{E}\sup\limits_{0\leq T\leq
\tau^*\wedge\mathcal{T}}\|h(T)\|^p+C_p\int_0^{
\mathcal{T}}\mathbb{E}\sup\limits_{0\leq r\leq
\tau^*\wedge\tau}\|h(r)\|^{p}d\tau\nonumber\\
&&+C_p\mathbb{E}\left[\int_0^{\tau^*\wedge\mathcal{T}}\|R(\tau)\|^{2p}d\tau\right]^\frac{1}{2}.\label{lemma-9-2}
\end{eqnarray}
Therefore, collecting together \eqref{lemma-7-bound},
\eqref{lemma-8-bound}, \eqref{lemma-9-1} and \eqref{lemma-9-2}, for
any $\mathcal{T}\in [0, T_0]$ we obtain
\begin{eqnarray}
\mathbb{E}\sup\limits_{0\leq
T\leq\tau^*\wedge\mathcal{T}}\|h(T)\|^{p}&\leq&
C_p\int_0^{\mathcal{T}}\mathbb{E}\sup\limits_{0\leq
r\leq\tau^*\wedge\tau}\|h(r)\|^pd\tau\nonumber\\
&&+C_p\varepsilon^{p-18\kappa p}.\nonumber
\end{eqnarray}
Using  Gronwall's lemma we can show
\begin{eqnarray}
\mathbb{E}\sup\limits_{0\leq T\leq \tau^*}\|h(T)\|^p&\leq&
C_p\varepsilon^{p-18\kappa p}+\varepsilon^{p-9\kappa p}\|a(0)\|^p
\nonumber\\
&\leq& C_p\varepsilon^{p-18\kappa p}, \nonumber
\end{eqnarray}
so that, in view of \eqref{lemma-7-bound},
\begin{eqnarray}
\mathbb{E}\sup\limits_{0\leq T\leq
\tau^*}\|a(T)-b(T)\|^p&\leq&\mathbb{E}\sup\limits_{0\leq T\leq
\tau^*}\|h(T)\|^p+\mathbb{E}\sup\limits_{0\leq T\leq
\tau^*}\|R(T)\|^p\nonumber\\
&\leq&C_p\varepsilon^{p-18\kappa p}.\nonumber
\end{eqnarray}
\end{proof}

\begin{remark}
Notice that by Lemma \ref{lemma-9}, for  any $p>0$ and $\kappa\in
(0,\frac{1}{18})$, we obtain
\begin{eqnarray}
\mathbb{E}\sup\limits_{0\leq T\leq \tau^*}\|a(T) \|^p&\leq&
\mathbb{E}\sup\limits_{0\leq T\leq
\tau^*}\|a(T)-b(T)\|^p+\mathbb{E}\sup\limits_{0\leq T\leq \tau^*}\|
b(T)\|^p\nonumber\\
&\leq&C_p(1+\|a(0)\|^p)\label{Remark2.1-1}.
\end{eqnarray}
We can use this to show that $\|a\|<\varepsilon^{-\kappa}$ on $[0,T_0]$ with probability almost $1$.
\end{remark}

Let us define the overall error between $\varepsilon(b+Q)$ and $u$ by
\begin{eqnarray}
\mathcal{R}(T)&:=&u(\varepsilon^{-2}T)-\varepsilon b(T)-\varepsilon Q(T)\nonumber\\
&=& \varepsilon [a(T)-b(T)+\psi(T)-Q(T)]\nonumber\\
&=& \varepsilon [a(T)-b(T)+I(T)+J(T)+K(T)] . 
\label{Remainder2}
\end{eqnarray}
We already know that $I$ and $J$ are uniformly small. It remains to bound $K$.
For this we use the factorization method and start with the following stochastic integral.
\begin{lemma}\label{small-stoch-integral}
Assume the setting of Lemma
\ref{lemma-9}. 
For any $p>1$, there exists a constant $C_p>0$ such
that
\begin{eqnarray}
\mathbb{E}\sup\limits_{0\leq T\leq
T_0}\Big\|\int_0^Te^{\mathcal{A}_s(T-\tau)\varepsilon^{-2}}G_s'(0)\cdot
b(\tau)d\tilde{W}_\tau \Big \|^p_\alpha \leq
C_p\varepsilon^{p-\frac{1}{2}\kappa}. \label{small-stoch}
\end{eqnarray}
\end{lemma}
\begin{proof}
For large $p>1$ fix $\gamma \in (0,1/2)$.
By the celebrated factorization method, if we set
\[
Y_\gamma(s) = \int_0^s(s-\tau)^{-\gamma}  e^{\mathcal{A}_s
(s-\tau)\varepsilon^{-2}}G_s'(0)\cdot b(\tau) d\tilde{W}_\tau
\]
we have 
\[
\int_0^T e^{\mathcal{A}_s(T-\tau)\varepsilon^{-2}} G_s'(0)\cdot
b(\tau) d\tilde{W}_\tau = C_\gamma\int_0^T  (T-\tau)^{\gamma-1}
e^{\mathcal{A}_s(T-\tau)\varepsilon^{-2}}Y_\gamma(\tau) d\tau
\]
for  some constant $C_\gamma>0$. Thus we obtain by H\"older
inequality and the bounds on the semigroup on the space $\mathcal{S}$ that
\begin{eqnarray*}
\lefteqn{ \Big\|\int_0^T e^{\mathcal{A}_s(T-\tau)\varepsilon^{-2}} G_s'(0)\cdot b(\tau) d\tilde{W}_\tau \Big\|_\alpha^p} \\
 &\leq & C \Big( \int_0^T  (T-\tau)^{-(1-\gamma)} e^{-\rho(T-\tau)\varepsilon^{-2}}\|Y_\gamma(\tau)\|_\alpha  d\tau \Big)^p \\
 &\leq & C  \Big(  \int_0^T  (T-\tau)^{-(1-\gamma) p/(p-1)} e^{-\rho(T-\tau)\varepsilon^{-2}p/(p-1)} d\tau \Big)^{p-1}
 \int_0^T \|Y_\gamma(\tau)\|_\alpha^p  d\tau \\
 &\leq &C  \varepsilon^{2\gamma p-2}   \Big(\int_0^{T/\varepsilon^2}\tau^{-{(1-\gamma)p}/{(p-1)}} e^{-\rho\tau p /{(p-1)}} d\tau \Big)^{p-1}\int_0^T \|Y_\gamma(\tau)\|_\alpha^p
 d\tau\\
 &\leq& C \varepsilon^{2\gamma p-2}   \int_0^T \|Y_\gamma(\tau)\|_\alpha^p  d\tau
 \;.     
\end{eqnarray*}
Note that we need to fix $p\gg1$ large or $1\gg\gamma>0$ small in order to have an integrable pole in the previous estimate. 

Moreover, by Burkholder-Davis-Gundy inequality, but now without the supremum in time 
we obtain for $t\in[0,T_0]$
\begin{eqnarray*}
 \mathbb{E} \|Y_\gamma(t)\|_\alpha^p
 &\leq & C\mathbb{E} \Big( \int_0^t(t-\tau)^{-2\gamma} \Big\|e^{\mathcal{A}_s (t-\tau)\varepsilon^{-2}}G_s'(0)\cdot b(\tau)
 \Big\|^2_{\mathscr{L}_2(U, \mathcal {H}^\alpha)} d\tau \Big)^{p/2}\\
 &\leq &  C\mathbb{E}\Big( \int_0^t(t-\tau)^{-2\gamma}  e^{-\rho(t-\tau)\varepsilon^{-2}}\|b(\tau)\|^2 d\tau \Big)^{p/2}\\
&\leq & C\mathbb{E} \sup\limits_{0\leq \tau \leq T_0} \|b(\tau)\|^p  \varepsilon^{2(1-2\gamma) p/2} \\
&\leq & C\mathbb{E} \sup\limits_{0\leq \tau \leq T_0} \|b(\tau)\|^p  \varepsilon^{(1-2\gamma) p}\;.
\end{eqnarray*}
This finally implies
\[
\mathbb{E} \sup\limits_{0\leq T\leq T_0} \Big\|\int_0^T
e^{\mathcal{A}_s(T-\tau)\varepsilon^{-2}} G_s'(0)\cdot b(\tau)
d\tilde{W}_\tau \Big\|_\alpha^p
 \leq C T_0  \mathbb{E} \sup\limits_{0\leq \tau \leq
T_0} \|b(\tau)\|^p \varepsilon^{p-2}\;.
\]
 As we can choose $p$ arbitrarily large, we obtain via H\"older inequality that for any
small $\tilde\kappa>0$ there is a constant such that
\begin{eqnarray*}
\mathbb{E} \sup_{0\leq T\leq T_0} \Big\|\int_0^T
e^{\mathcal{A}_s(T-\tau)\varepsilon^{-2}} G_s'(0)\cdot b(\tau)
d\tilde{W}_\tau \Big\|_\alpha^p \leq C   \mathbb{E} \sup_{0\leq \tau
\leq T_0} \|b(\tau)\|^p \varepsilon^{p - \tilde\kappa}\;,
\end{eqnarray*}
so that, thanks to \eqref{lemma-8-bound}, we have
\begin{eqnarray*}
\mathbb{E} \sup_{0\leq T \leq T_0} \Big\|\int_0^T
e^{\mathcal{A}_s(T-\tau)\varepsilon^{-2}} G_s'(0)\cdot b(\tau)
d\tilde{W}_\tau \Big\|_\alpha^p \leq
C_p\varepsilon^{p-\frac{1}{2}\kappa}.
\end{eqnarray*}
\end{proof}
In order to bound $K$ we set
\begin{eqnarray*}
M(T):=\int_0^Te^{\mathcal{A}_s(T-\tau)\varepsilon^{-2}}\left(\frac{1}{\varepsilon}G_s(\varepsilon
a(\tau)+\varepsilon\psi(\tau))- G_s'(0)\cdot
b(\tau)\right)d\tilde{W}_\tau,
\end{eqnarray*}
we have from \eqref{e:DB1}
\begin{eqnarray*}
K(T)= M(T)+\int_0^Te^{\mathcal{A}_s(T-\tau)\varepsilon^{-2}}G_s'(0)\cdot b(\tau) d\tilde{W}_\tau.
\end{eqnarray*}
where we just bounded the integral on the right in Lemma \ref{small-stoch-integral}. 
It remains to bound $M$. Here we proceed similarly to the previous lemma using factorization.
\begin{lemma}\label{lemma-10}
Assume the setting of Lemma \ref{lemma-9}. For any $p>1$ there
exists a constant $C_p>0$ such that
\begin{eqnarray}
\mathbb{E}\sup\limits_{0\leq T\leq \tau^*}\|M(T)\|_\alpha^p\leq
C_p\varepsilon^{p-2\kappa p}.
\end{eqnarray}
\end{lemma}
\begin{proof}

We can follow exactly the proof of the previous Lemma \ref{small-stoch-integral} 
but have to pay attention to the fact that the integrand in $M$ is only defined for $t\leq\tau^*$. 
Moreover, the integrand is due to the presence of $\psi$ and thus $K$ not uniformly bounded in time.

Define the integrand as 
\[
\Phi(\tau)= \frac{1}{\varepsilon}G_s(\varepsilon
a(\tau)+\varepsilon\psi(\tau))- G_s'(0)\cdot
b(\tau)
\]
We notice that by Taylor's formula
\begin{eqnarray}
\Phi(\tau)&=& G'(0)\cdot \psi(\tau)+G'(0)\cdot[a(\tau)-b(\tau)]\nonumber\\
&&\quad+\frac{\varepsilon}{2}G''(\tilde{z}(\tau))\cdot( a(\tau)+
\psi(\tau), a(\tau)+ \psi(\tau)),\nonumber
\end{eqnarray}
where $\tilde{z}(\tau)$ is a vector on the line segment connecting
$0$ and $\varepsilon a(\tau)+\varepsilon\psi(\tau)$. Therefore,
\[
\Big\|\Phi(\tau)\Big\|_{\mathscr{L}_2(U, \mathcal{H}^\alpha)}
\leq  C\|\psi(\tau)\|_\alpha+C\|a(\tau)-b(\tau)\|_\alpha
+C\varepsilon\|a(\tau)\|_\alpha^2+C\varepsilon\|\psi(\tau)\|_\alpha^2,
\]
Note that for $\tau\leq\tau^*$ the right hand side above is bounded by $C\varepsilon^{-\kappa}$ uniformly in time.

We obtain following the lines of  Lemma \ref{small-stoch-integral}
\begin{eqnarray}
\lefteqn{\mathbb{E}\sup\limits_{0\leq T\leq \tau^*}\|M(T)\|_\alpha^p} \\
&& \leq  
\mathbb{E} \sup\limits_{0\leq T\leq T_0}  
\Big\|
\int_0^T  e^{\mathcal{A}_s(T-\tau)\varepsilon^{-2}}   1_{[0,\tau*]}(\tau)\Phi(\tau) d\tilde{W}_\tau
\Big\|_\alpha^p \nonumber \\
&&
\leq  C \varepsilon^{p-\kappa}
\mathbb{E} \sup\limits_{0\leq T\leq T_0}  
\Big\| 1_{[0,\tau*]}(\tau)\Phi(\tau)
\Big\|_{\mathscr{L}_2(U,\mathcal{H}^\alpha)}^p \nonumber\\
&&
\leq  C \varepsilon^{p-\kappa}
\mathbb{E} \sup\limits_{0\leq T\leq \tau*}  
\Big\|  \Phi(\tau)
\Big\|_{\mathscr{L}_2(U,\mathcal{H}^\alpha)}^p \nonumber
\end{eqnarray}
so that
\[
\mathbb{E}\sup\limits_{0\leq T\leq \tau^*}\|M(T)\|_\alpha^p \leq C\varepsilon^{p(1-2\kappa)}
\] 
\end{proof}

As a consequence of Lemmas \ref{lemma-1}, \ref{lemma-9}, and
\ref{lemma-10}, we have the following bound on $\mathcal{R}$:
\begin{lemma}\label{lemma-10-1}
Assume the setting of Lemma \ref{lemma-9}. For any $p>1$, there
exists a constant $C_p>0$ such that
\begin{eqnarray}
\mathbb{E}\sup\limits_{0\leq T\leq \tau^*}\|\mathcal
{R}(T)\|_\alpha^p\leq C_p\varepsilon^{2p-18\kappa
p}.\label{lemma-10-1-bound}
\end{eqnarray}
\end{lemma}
Moreover, we obtain a bound on $\psi$ which is uniform in time:
\begin{lemma}\label{lemma-12}
Assume the setting in Lemma \ref{lemma-9}. For any $p>1$, there
exists a constant $C_p>0$ such that
\begin{eqnarray}
\mathbb{E}\sup\limits_{0\leq T\leq \tau^*}\| \psi(T)\|^p_\alpha\leq
C_p(\varepsilon^{-p\kappa/3}+\varepsilon^{p-18\kappa
p}).\label{lemma-12-bound}
\end{eqnarray}
\end{lemma}

Before we proceed with the final error estimate, we comment on  improved bounds on $\psi$ and $M$.
We know by definition that 
\[
\psi(T) = Q(T) + \int_0^Te^{\mathcal{A}_s(T-\tau)\varepsilon^{-2}}G_s'(0)\cdot b(\tau) d\tilde{W}_\tau + M(T) + I(T)+J(T) \;
\]
As $I(T)+J(T)=\mathcal{O}(\varepsilon^{2-3\kappa})$ by Lemma \ref{lemma-1} and both $M$ and the stochastic integral 
is uniformly in time by $\mathcal{O}(\varepsilon^{1-2\kappa})$, we can show that $\psi-Q$ is small uniformly in time.
This improved bound on $\psi$ can be used in the proof of Lemma \ref{lemma-10}, 
to show that $M$ is smaller of order  $\mathcal{O}(\varepsilon^{1-2\kappa})$.
Thus we obtain 
\begin{equation}
 \mathbb{E}\sup_{T\in[0,\tau^*]} \Big\| \psi(T) 
 - Q(T) - \int_0^Te^{\mathcal{A}_s(T-\tau)\varepsilon^{-2}}G_s'(0)\cdot b(\tau) d\tilde{W}_\tau \Big\|_\alpha^p
 \leq C\varepsilon^{p(2-3\kappa)}\;.
\end{equation}

Before proving main theory, we need to construct a subset of
$\Omega$, which enjoys nearly full probability.

\begin{definition}\label{Omega-*}
For $\kappa>0$ from the definition of $\tau^*$ define the set $\Omega^*\subset \Omega$ of
all $\omega\in \Omega$ such that all these estimates
\[
\sup\limits_{0\leq T\leq\tau^*}\|a(T)\|< \varepsilon^{-\frac{1}{2}\kappa}, \qquad 
\sup\limits_{0\leq T\leq\tau^*}\|\mathcal{R}(T)\|<\varepsilon^{2-19\kappa}
\]
and
\begin{eqnarray*}
\sup\limits_{0\leq
T\leq\tau^*}\|\psi(T)\|_\alpha<\varepsilon^{-\frac{1}{2}\kappa}.
\end{eqnarray*}
hold.
\end{definition}

\begin{lemma}
The set $\Omega^*$ has approximately probability 1.
\end{lemma}

\begin{proof}
Indeed, let $\Omega^*$ be as in the Definition \ref{Omega-*}. It
easily follows that
\begin{eqnarray*}
\mathbb{P}(\Omega^*)&\geq& 1-\mathbb{P}(\sup\limits_{0\leq
T\leq\tau^*}\|a(T)\|\geq\varepsilon^{-\frac{1}{2}\kappa})
-\mathbb{P}(\sup\limits_{0\leq
T\leq\tau^*}\|\psi(T)\|_\alpha\geq\varepsilon^{-\frac{1}{2}\kappa})\\
&&-\mathbb{P}(\sup\limits_{0\leq
T\leq\tau^*}\|\mathcal{R}(T)\|\geq\varepsilon^{2-19\kappa}).
\end{eqnarray*}
We get by using Chebychev's inequality,
\eqref{Remark2.1-1}, \eqref{small-stoch}, \eqref{lemma-10-1-bound}
and \eqref{lemma-12-bound}
\begin{eqnarray*}
\mathbb{P}(\Omega^*) 
&\geq&1 -
(\varepsilon^{-\frac{1}{2}\kappa})^{-q}
\mathbb{E}\sup\limits_{[0,\tau^*]}\|a(T)\|^q_\alpha
-(\varepsilon^{-\frac{1}{2}\kappa})^{-q}
\mathbb{E}\sup\limits_{[0,\tau^*]}\|\psi(T)\|^q_\alpha\\
&&-(\varepsilon^{2-19\kappa})^{-q}
\mathbb{E}\sup\limits_{[0,\tau^*]}\| \mathcal{R}(T)\|^q_\alpha
\\
&\geq& 1- 2C\varepsilon^{\kappa q /2} \varepsilon^{{-q\kappa/3}} -C \varepsilon^{(19\kappa-2)q}\varepsilon^{(2-18\kappa)q }\\
&\geq&1-C\varepsilon^{p}
\end{eqnarray*}
where we take for a given $p$ the exponent $q$ sufficiently large.
\end{proof}

\begin{proof}[Proof of Theorem 2.1]
From the definition of $\Omega^*$ and $\tau^*$, we have
\begin{eqnarray}
\Omega^*\subseteq \Big\{\sup\limits_{0\leq T\leq \tau^*}
\|a(T)\|_\alpha<\varepsilon^{-\kappa},\sup\limits_{0\leq T\leq
\tau^*} \|\psi(T)\|_\alpha<\varepsilon^{-\kappa}\Big\}\subseteq
\left\{\tau^*=T_0\right\}\subseteq\Omega.\nonumber
\end{eqnarray}
This allows us to get on $\Omega^*$
\[
\sup\limits_{0\leq T\leq T_0}\|\mathcal{R}(T)\|_\alpha
=\sup\limits_{0\leq T\leq\tau^*}\|\mathcal{R}(T)\|_\alpha
\leq C\varepsilon^{2-19\kappa}
\]
such that
\begin{eqnarray}
\mathbb{P}\Big(\sup\limits_{0\leq T\leq T_0}\|\mathcal{R}(T)\|_\alpha\geq
\varepsilon^{2-19\kappa}\Big)\leq 1-\mathbb{P}(\Omega^*)\leq
C\varepsilon^p,
\end{eqnarray}
which, by recalling  representation \eqref{Remainder2}, completes
the proof.
\end{proof}

\section{Example - Ginzburg-Landau/Allen-Cahn equation}
A very simple example to illustrate the main result is the stochastic Ginzburg-Landau  equation
(or Allen-Cahn equation) with linear multiplicative noise on the
interval $D=[0, \pi]$ of the form
\begin{eqnarray}
\partial_tu=(\partial_x^2+1)u+\nu\varepsilon^2u-u^3+\varepsilon
u\cdot\partial_tW(t). \label{example-1}
\end{eqnarray}
In the following we consider the It\^o-representation 
of the the SPDE above in the Sobolev-space $H^1(D)$ with sufficiently smooth noise. 

We set
$$
\mathcal{A}:=\partial_x^2+1,\quad
\mathcal{L}:=\nu \mathcal {I},\quad
\mathcal{F}:=-u^3
$$
Suppose that the equation is subjected to the Dirichlet boundary
condition. Let $\mathcal{H}=L^2([0, \pi ])$ be the space of all
square integrable real-valued functions which are defined on the
interval $[0, \pi]$. In this situation the eigenvalues of
$-\mathcal{A}$ are explicitly known to be
 $\lambda_k=k^2-1$  with  associated eigenvectors
$e_k(x)=\sqrt{\frac{2}{\pi}}\sin(kx)=\delta\sin(kx)$,
$k=1,2,\cdots$, and $\mathcal{N}=span\{sin\}$.
So Assumption 1 is true with $m=2$.

Clearly, Assumption 2 holds true for example for any $\alpha>1/4$ and
$\beta=0$, as for the norm in $\mathcal{H}^\alpha$
we then have $\|uv\|_\alpha \leq C\|u\|_\alpha\|v\|_\alpha$.
We will fix $\alpha=1$ for simplicity.

Note that on the one-dimensional space $\mathcal{N}$ the $\mathcal{H}^\alpha$-norm
is just a multiple of the $\mathcal{H}$-norm. So that
for $u,w \in\mathcal{N}$ the conditions described in Assumption 3 are
satisfied as follows:
\[
\langle\mathcal{F}_c(u), u\rangle=-\int_0^\pi\! u^4(x) dx \leq 0,
\quad\!
\langle\mathcal{F}_c(u,u,w),
w\rangle=-\int_0^\pi\! u^2(x)w^2(x) dx\leq 0.
\]
In addition, condition \eqref{F-condition-4} is true for some
positive constants $C_0, C_1$ and $C_2$, as $\mathcal {F}$ is a
standard cubic nonlinearity.

Define $f_k(x):=\frac{1}{ k}e_k(x),\;k=1,2,\cdots$,
such that $\{f_k\}_{k\in\mathbb{N}}$ is an
orthonormal basis of $\mathcal{H}^1$.
We consider in our application
that $W$ is standard cylindrical $\mathcal{H}^1$-valued Wiener process and define a
covariance operator $Q$ defined by
$Qf_k={\alpha_k}f_k,\;k=1,2,\cdots,$ satisfying
$\text{trace}(Q)=\sum\limits_{k=1}^\infty\alpha_k=C_0<\infty.$ For the operator
$G$ defined as
\[
G(u)\circ v:=u\cdot Q^{1/2}v\;,
\]
we have
\begin{eqnarray*}
\|G(u)\|^2_{\mathscr{L}^2({\mathcal{H}^1}, \mathcal{H}^1)}
&=&\sum\limits_{k\in\mathbb{N}}\|u\cdot(Q^\frac{1}{2}e_k)\|_{\mathcal{H}^1}^2
=\sum\limits_{k\in\mathbb{N}}\alpha_k\|u \cdot e_k\|^2_{\mathcal{H}^1}\\
&\leq&
C\sum\limits_{k\in\mathbb{N}}\alpha_k\|u\|_{\mathcal{H}^1}^2\|e_k\|_{\mathcal{H}^1}^2
=C\|u\|_{\mathcal{H}^1}^2 \text{trace}(Q) \\
&\leq& C\|u\|_{\mathcal{H}^1}^2 <\infty.
\end{eqnarray*}
Therefore, $G(\cdot): \mathcal{H}^1\rightarrow
\mathscr{L}^2({\mathcal{H}^1}, \mathcal{H}^1)$ is a Hilbert-Schmidt operator satisfying $G'(u)\cdot v = v \cdot Q^{1/2}$ and $G''(u)= 0$, so that
Assumption 5 holds.

Therefore, our main theorem states that the dynamics of
\eqref{example-1} can well approximated by  the amplitude equation, which is for $b\in\mathcal{N}$ a stochastic ordinary
differential equation of the form:
\[
db = [\nu b- P_c\mathcal{F}(b)] dt + P_c [G'(0)\cdot b] dW \;.
\]
Let us finally rewrite the the amplitude equation for the actual amplitude of $b$
\[
b=\gamma\sin
\]
Clearly, as $P_cf = \frac2\pi \int_0^\pi \sin(y)f(y)dy \sin$ we have  $P_c\mathcal{F}(b) = - \frac{3}{4}\gamma^3 \sin$. 
Moreover,
\begin{eqnarray*}
P_c [G'(0)\cdot b] dW &=&  \gamma  \sum_{k=1}^\infty  P_c[ \sin Q^{1/2} f_k] d\tilde{B}_k
\\ &=&
\gamma  \sum_{k=1}^\infty  \sqrt{\alpha_k} P_c[ \sin f_k] d\tilde{B}_k
= \gamma  \sum_{k=1}^\infty  \sqrt{\alpha_k} \sigma_k  d\tilde{B}_k \sin
\end{eqnarray*}
as $P_c[\sin(x)\sin(kx)] = \frac2\pi \int_0^\pi \sin(y)^2 \sin(ky) dy \sin(x) = k \sigma_k \sin(x)$
with
\[
\sigma_k:= \left\{
\begin{array}{ccl}
 \frac4\pi \frac{\cos(k\pi)-1}{k^2 (k^2-4)}&:&  k\neq 2,\\
0 &:& k=2.
\end{array}
\right.
\]
Thus we obtain for the amplitude
\begin{eqnarray*}
d\gamma=[\nu\gamma-\frac{3}{4}\gamma^3] dT + \gamma \Sigma^{1/2} d\beta,
\end{eqnarray*}
for a standard real valued Brownian motion and noise strength $\Sigma = \sum\limits_{k=1}^\infty\alpha_k\sigma_k^2$.

\section*{Acknowledgments}
Hongbo Fu would like to thank Professor Dirk Bl\"{o}mker for the
hospitality during his visit in 2016 and CSC for funding this
visit. Hongbo Fu is supported by Natural Science Foundation of
Hubei Province (No. 2018CFB688) and  NSF of China (Nos. 11826209,
11301403, 11971367).










\end{document}